\documentclass[12pt]{amsart}

\usepackage{amsmath,amsthm,amsfonts,amssymb,hyperref}
\usepackage[margin=1in]{geometry}    
\usepackage{parskip}

\newtheorem{theorem}{Theorem}[section]
\newtheorem{lemma}{Lemma}[section]

\newtheorem{prop}{Proposition}[section]
\newtheorem{remark}{Remark}[section]
\numberwithin{equation}{section}

\newcommand{\R}{\mathbb{R}}
\newcommand{\Z}{\mathbb{Z}}

\newcommand{\E}{\mathbb{E}}

\makeatletter
\@namedef{subjclassname@2020}{%
  \textup{2020} Mathematics Subject Classification}
\makeatother

\begin{document}

\title[DISCRETE MOVING POLYMER]{EFFECTIVE RADIUS OF A DISCRETE MOVING POLYMER}
\author{Jiaming Chen}
\address{Dept. of Mathematics
\\Imperial College
\\London, UK}
\email{j.chen1@imperial.ac.uk}
\begin{abstract}
We present a discrete model for a weakly self-avoiding polymer of intrinsic length $J$ governed by a discrete heat equation perturbed by Gaussian noise. By introducing a renormalized penalizing factor, we establish sharp bounds on the polymer's root mean squared radius of gyration $R(T,J)$ scales linearly with $J$, up to logarithmic corrections and explicit powers of the self-repellence parameter $\beta$. This is consistent with the weak-interaction results for equilibrium polymers, highlighting both the distinct nature of the discrete model and its analytical tractability for studying polymer
behavior in lattice-like environments.
\end{abstract}

\maketitle

\section{Introduction}
Over the past few decades, random polymers have attracted significant attention across disciplines including chemistry, physics, and mathematics, due to their ubiquity in the physical world. These complex structures play a central role in areas such as molecular theory, statistical mechanics, and probability theory. For a comprehensive overview from the physical sciences' perspective, we refer to \cite{doi1988theory}, while rigorous mathematical treatments can be found in \cite{hollander2009random} and \cite{bauerschmidt2012lectures}. A detailed survey of key results and open problems in the one-dimensional setting is provided by \cite{van2001survey}.

A polymer is a molecule composed of many ``monomers'' (groups of atoms) linked together through chemical bonds. From a mathematical perspective, modeling polymers is challenging due to the complex interactions that influence their spatial configurations. The most basic model treats the polymer as a simple random walk, where the time parameter represents the progression along the chain, and each new segment is added in a random direction with identical step distributions. However, this model fails to capture an essential physical constraint: two monomers cannot occupy the same position in space, and in fact, tend to repel one another when in proximity. This phenomenon, known as the \textit{excluded volume effect}, motivates a fundamental modification by introducing a penalty for self-intersections. Extensive research has been conducted on this topic; see \cite{madras2013self} for a detailed account. 

The \textit{Rouse model} \cite{rouse1953theory} serves as a foundational framework for studying the dynamics of polymer solutions without self-avoidance, modeling the polymer as a chain of beads connected by harmonic springs. Let $(R_1,R_2,...,R_N)$ be the position (vectors) of the beads, then for internal beads $(n=2,3,...,N-1)$,
\begin{equation*}
\zeta\frac{dR_n}{dt} = \kappa(R_{n+1}-2R_n+R_{n-1})+f_n.
\end{equation*}
For the end beads $(n=1$ and $n=N)$,
\begin{equation*}
\zeta\frac{dR_1}{dt} = \kappa(R_{2}-R_{1})+f_1,~\zeta\frac{dR_N}{dt} = \kappa(R_{N-1}-R_{N})+f_N,
\end{equation*}
where $\zeta,\kappa$ are constants and $f_n$ is a Gaussian random force indexed on $n$.

Doi and Edwards \cite{doi1988theory} obtain the following stochastic partial differential equation (SPDE) as a continuum limit of the Rouse model. This equation, also known as the \textit{Edwards-Wilkinson model}, takes the form:
\begin{equation*}
\begin{split}
\partial_tu&=\partial^2_xu+\dot{W}(t,x),\\
u(0,x) &= u_0(x),
\end{split}
\end{equation*}
where $(\dot{W}(t,x))_{t\ge0,x\in[0,J]}$ is a two-parameter white noise and $u_0$ is a continuous and bounded initial function. The authors argue that, over long time scales, key universal properties of the polymer, such as scaling behavior, should be consistent between the discrete and continuous models. 

To the best of the author's knowledge, Mueller and Neuman \cite{mueller2022scaling} were the first to rigorously analyze the scaling behavior of a dynamically evolving and weakly self-avoiding polymer that incorporates time as a variable. In reality, even when the underlying model is continuous, we typically observe it only at discrete points in space or time. While discretizing continuous models is standard in both theoretical and practical settings, such discretizations require careful treatment. Verifying that the discrete model preserves the expected scaling behavior thus requires new techniques and careful analysis. The goal of this paper is to study a discrete-time and discrete-space version of the polymer model in one dimension under suitable assumptions.

We now consider the following discretized model:
\begin{equation}
\label{discModel}
\begin{split}
 u(t+1,n)-u(t,n) &= \kappa\left(u(t,n+1)-2u(t,n)+u(t,n-1)\right)+\xi(t,n),\\
 u(0,n) &= u_0(n),
\end{split}
\end{equation}
which describes the temporal evolution of the function $u(t,n)$ under a discrete Laplacian operator perturbed by a standard Gaussian random variable $\xi(t,n)$. To ensure convergence and stability of the numerical scheme, we set the diffusion coefficient to $\kappa = \frac{1}{2}$. Without loss of generality, the analysis is restricted to grid points within the domain $[0,T]\times[0,J-1]$, corresponding to the index sets $\mathcal{T}:=\{0,1,...,T\}$ and $\mathcal{J}:=\{0,1,...,J-1\}$ for $T,J\in\Z^+$. Here $u_0$ is the initial profile bounded uniformly in $J$. In the absence of external interactions, we assume the noise variables are independent and characterized by the following moments:
\begin{equation}
\label{noise}
\begin{split}
\E[\xi(t,n)] &= 0,\\
\E[\xi(t,n)\xi(s,m)]&=\delta_{ts}\delta_{nm},
\end{split}
\end{equation}
where $\delta$ denotes the Kronecker delta. Since the endpoints of the polymer are not fixed, we impose discrete Neumann boundary conditions at the hypothetical beads $u(\cdot,-1)$ and $u(\cdot,J)$:
\begin{equation}
\label{NeuCond}
u(\cdot,-1) = u(\cdot,0),~u(\cdot,J) = u(\cdot, J-1).
\end{equation}

The heat kernel corresponding to equation \eqref{discModel} under the Neumann boundary in \eqref{NeuCond} is given by:
\begin{equation*}
    G_t(n,k)=\sum_{m=0}^{J-1}a_m\left(\cos\left(\frac{m\pi}{J}\right)\right)^{t}\cos\left(\frac{m\pi(n+0.5)}{J}\right)\cos\left(\frac{m\pi(k+0.5)}{J}\right),
\end{equation*}
where $a_m = \begin{cases}
\frac{1}{J}& m=0,\\
\frac{2}{J}&m=1,...,J-1
\end{cases}$.
Let us denote 
\begin{equation}
\label{kernel_notation}
\rho_m =\cos\left(\frac{m\pi}{J}\right),~~\phi_m(n)=\begin{cases}
\sqrt{\frac{1}{J}}& m=0,\\
\sqrt{\frac{2}{J}}\cos\left(\frac{m\pi(n+0.5)}{J}\right)&m=1,...,J-1,
\end{cases}
\end{equation}
and the solution $u(t,n)$ is given by
\begin{equation}
\label{sol}
    u(t,n)=\sum_{m=0}^{J-1}G_t(n,m)u_0(m) + \sum_{s=0}^{t-1}\sum_{m=0}^{J-1}G_{t-s}(n,m)\xi(s,m).
\end{equation}

We now incorporate self-avoidance into our discrete polymer model. A classical approach in statistical mechanics is to modify the reference probability measure by assigning a penalty proportional to the number of self-intersections. Specifically, we weight each path by the exponential of the negative product of a parameter and the pairwise self-intersection count. These weights are often referred to as \textit{Boltzmann weights} in statistical mechanics.

In our setting, since the solution $u(t,n)\in\R$ while the index $n$ belongs to a discrete set, we approximate the self-intersection by introducing a finite spatial resolution, following a similar approach to that of \cite{lin2025radius}. For each time $t$, we define the approximate self-intersection count as
\begin{equation}
\label{localTime}
\begin{split}
    N_d(t):=\sum_{0\leq i\neq j\leq J-1}1_{\vert u(t,i)-u(t,j)\vert\leq d}.
    \end{split}
\end{equation}
This expression quantifies the number of pairs of monomers (indexed by $i$ and $j$) that are within $d$ distance of each other at time $t$, capturing the local crowding of the polymer. One might consider defining the self-interaction over the entire time horizon via the total sum
\begin{equation*}
    \sum\limits_{t=1}^{T}\sum_{0\leq i\neq j\leq J-1}1_{\vert u(t,i)-u(t,j)\vert\leq d}.
\end{equation*}
However, following the rationale in \cite{mueller2022scaling}, there is no physical reason why different parts of the polymer cannot occupy the same spatial location at different times. Therefore, it is more appropriate to work with the time-slice-based definition in \eqref{localTime}.

Let $P_{T,J}$ denote the uniform probability measure on the space of polymer configurations $(u(t,n))_{t\in\mathcal{T},n\in\mathcal{J}}$. \cite{lin2025radius} investigated the case when $d\equiv 1$. As $d$ varies, it is natural to renormalize the self-intersection count by a factor of $d^\alpha$, for some $\alpha\geq 0$. We define the penalizing factor
\begin{equation}
\label{penalF}
    \mathcal{E}_{T,J,\beta,d,\alpha}=\exp\left(-\frac{\beta}{d^\alpha}\sum_{t=1}^{T}N_d(t)\right),
\end{equation}
and define the corresponding self-repelling polymer measure $Q_{T,J,\beta,d,\alpha}$ via
\begin{equation}
\label{Qmeas}
Q_{T,J,\beta,d,\alpha}(A) =\frac{ \mathbb{E}^{P_{T,J}}[1_{A}\mathcal{E}_{T,J,\beta,d,\alpha}]}{Z_{T,J,\beta,d,\alpha}}, \text{ for all measurable set $A$},
\end{equation}
where the partition function $Z_{T,J,\beta,d,\alpha}$ is given by
\begin{equation}
\label{Z}
Z_{T,J,\beta,d,\alpha} = \mathbb{E}^{P_{T,J}}[\mathcal{E}_{T,J,\beta,d,\alpha}].
\end{equation}

The $\beta>0$ is a parameter governing the strength of self-repellence and typically represents the inverse temperature of the system. When $d=1$, the probability measure $Q_{T,J,\beta,d,\alpha}$ is known as the \textit{$J-$polymer measure with strength of self repellence $\beta$}, and equation \eqref{penalF} defines what is known as the \textit{Domb-Joyce model} for soft polymers (see \cite{madras2013self} Section 10.1). It models the polymer in equilibrium with itself and its environment. Each self-intersection at a fixed time incurs an energy cost of $\beta$, which translates into a penalty of $e^{-\beta}$ for that configuration. This induces an effective repulsion that causes the polymer to spread out over space. In this article, we study the case $d\geq 1$ and analyze how the model changes as $\alpha\to\infty$.

For notational convenience, we suppress subscripts except for the time horizon $T$, and write:
\begin{equation*}
P_{T} = P_{T,J},~Q_{T}=Q_{T,J,\beta,d,\alpha},~\mathcal{E}_{T} = \mathcal{E}_{T,J,\beta,d,\alpha},~Z_{T}=Z_{T,J,\beta,d,\alpha}.
\end{equation*}

Throughout this paper, we use $C$, $C'$ and $C''$ to denote a generic positive constant, whose value may change from line to line. Dependence of constants on specific parameters is explicitly indicated by listing the parameters in parentheses.

In Section \ref{main_result}, we state the main theorems concerning the spatial extent of the polymer under the self-repelling measure $Q_T$. Section \ref{Lower Bound on the Partition Function} derives a lower bound for the partition function $Z_T$ via a Gaussian exponential tilting. Sections \ref{Small distance tail estimate} and \ref{Large distance tail estimate} establish tail bounds for the effective radius, addressing the lower and upper deviations, respectively, thereby completing the proof of Theorem \ref{preTh}. Finally, Theorem \ref{mainTh} is proved in Section \ref{limitalpha} via Lemma \ref{interchange}.

\section{Main Result}
\label{main_result}
A fundamental quantity in characterizing the spatial configuration of a polymer is its macroscopic spatial extent. In this work, we focus on the \textit{radius of gyration}, a widely accepted measure defined as the mean squared distance of the polymer's segments from its center of mass. This quantity captures the average dispersion of the polymer in space and serves as a robust indicator of its effective size.

We define the center of mass of the polymer configuration at time $t$ by
\begin{equation}
\label{ubar}
\bar{u}(t) = \frac{1}{J}\sum_{n=0}^{J-1}u(t,n),
\end{equation}
and introduce the root mean squared radius of gyration over time horizon $T$ as
\begin{equation}
\label{RTJ}
R(T,J) = \left[\frac{1}{TJ}\sum_{t=1}^T\sum_{n=0}^{J-1}(u(t,n)-\bar{u}(t))^2\right]^{1/2}.
\end{equation}
This definition generalizes a classical concept discussed in Equation $(2.6)$ on page 9 of \cite{doi1988theory}, where the polymer's size is approximated by the end-to-end distance $|R_N-R_1|$ for a polymer $(R_1,R_2,...,R_N)$. However, as emphasized by the authors on page 22 of the same reference, this end-to-end distance is not always an appropriate measure, especially for branched polymers where it may be ill-defined. Even though our focus is on linear polymers, we adopt the radius of gyration $R(T,J)$ as it more comprehensively captures the spatial structure through fluctuations around the center of mass.

In the context of the one-dimensional Domb-Joyce model, Bolthausen \cite{bolthausen1990self}, using large deviation techniques, showed that the end-to-end distance grows linearly with polymer length $J$ almost surely only for small values of the self-repellence parameter $\beta$. The survey paper \cite{van2001survey} contains a much more complete picture, listing several results that establish the $J$-scaling with greater generality and precision. Particularly, the results are valid for all $\beta > 0$, not just small $\beta$. Both articles concern equilibrium polymers, i.e., models without a time component. 

More recently, Mueller and Neuman \cite{mueller2022scaling} showed in the continuous setting that the effective radius of a moving weakly self-avoiding polymer scales as $J^{5/3}$. The results of the present paper concern a time-dependent discrete model with a modified penalizing factor. Although the scaling behavior is consistent with known equilibrium results, our analysis provides the proof in a dynamical discrete framework.

\begin{theorem}
    \label{preTh}
    There exist positive constants $K_1$, $K_2$, $\tilde{J}$, $\tilde{d}$ and $\beta_C$ such that for $\frac{\tilde{d}J}{\log J}\geq d\geq 1$, $\beta \geq\beta_Cd^{4+\alpha}J^{-2}(\log J)^{2}$ and $J>\tilde{J}$, we have
    \begin{equation*}
        \lim_{T\to\infty}Q_T\left[K_1\beta_C^{\frac{1}{3}}d J\leq R(T,J)\leq K_2\beta^{\frac{1}{3}}d^{\frac{2-\alpha}{3}}J^{\frac{5}{3}}(\log J)^{-\frac{1}{6}}\right]=1.
    \end{equation*}
\end{theorem}
A particularly important choice is $d=\left(\frac{J}{\log J}\right)^{\frac{2}{4+\alpha}}$, for which the lower bound on $\beta$ becomes independent of both $d$ and $J$. Passing then to the limit $\alpha\to\infty$ leads to the following result.

\begin{theorem}
\label{mainTh}
    There exist positive constants $K_1$, $K_2$, $\tilde{J}$ and $\beta_C$ such that for $\beta\geq\beta_C$ and $J>\tilde{J}$,
    \begin{equation*}
        \lim_{T\to\infty}\widetilde{Q}_T\left[K_1\beta_C^{\frac{1}{3}}J\leq R(T,J)\leq K_2\beta^{\frac{1}{3}}J(\log J)^{\frac{1}{2}}\right]=1.
    \end{equation*}
    The corresponding penalizing factor is
    \begin{equation}
    \label{limit_penal}
        \widetilde{\mathcal{E}}_T=\exp\left(-\frac{\beta(\log J)^2}{J^2}\sum_{t=1}^TN_1(t)\right).
    \end{equation}
\end{theorem}

\begin{remark}
Most known results for weakly self-avoiding polymers do not produce an end-to-end distance with explicit dependence on the self-repellence parameter $\beta$. Our result is consistent with this general phenomenon. Indeed, at the level of the lower bound, the root mean squared radius $R(T,J)$ is asymptotically independent of $\beta$, whereas the upper bound retains explicit dependence on $\beta$, for all $\beta\in[\beta_C,\infty)$. 
\end{remark}

\begin{remark}
If no renormalization is introduced and $d$ is fixed in \eqref{penalF}, the model is comparable to those studied in \cite{lin2025radius} and \cite{mueller2022scaling}. In this case, the self-intersection count $N_d(t)$ includes the $J$ diagonal terms, which contribute a finite but non-negligible penalty independent of the spatial configuration of the polymer. In this regime, Theorem \ref{preTh} is suboptimal in the sense that a gap remains between the lower and upper bounds. Notably, the upper bound matches the scaling found in \cite{mueller2022scaling} for the continuous model.

By contrast, once $d$ depends on $J$ and renormalization is introduced with $\alpha \to \infty$ in \eqref{penalF}, the model enters a fundamentally different regime, with the penalizing factor given in \eqref{limit_penal}. The effective coupling $\tilde{\beta}=\beta/d^\alpha$ then converges, as $\alpha\to\infty$, to $\beta(\log J)^2/J^2$, which vanishes as $J \to \infty$ for fixed $\beta\geq\beta_C$. In this regime, the self-repellence yields the sharp $J$-scaling, with matching upper and lower bounds up to logarithmic factors. This is consistent with the weak-interaction results for equilibrium polymers established in \cite{hofstad2003weak}.
\end{remark}

\begin{remark}
    The proof of Theorem \ref{preTh} makes essential use of the Gaussian nature of the noise $\xi(t,n)$ in two places. First, the lower bound on the partition function $Z_T$ in Section \ref{Lower Bound on the Partition Function} relies on the fact that spatial increments $u(0,i)-u(0,j)$ are Gaussian under the tilted measure, which allows the self-intersection probability to be easily bounded. Second, the large deviation principle invoked in Section \ref{Large distance tail estimate} relies on the fact that the driving noise of each component $X_t^{(m)}$ is Gaussian, which enables the moment-generating function of the quadratic functional to be computed explicitly via iterated conditional expectations. For non-Gaussian noise, neither of these computations would be tractable in the same way, and extending the result to more general noise distributions remains an interesting open problem.
\end{remark}

\section{Lower Bound on the Partition Function}
\label{Lower Bound on the Partition Function}
In this section, we analyze the behavior of the partition function $Z_T$ and derive a rigorous lower bound. To this end, we construct a penalized path by introducing a drift term, following the methodology outlined in Section 3 of \cite{bolthausen1990self}, which we will define precisely below. This approach is standard: under the tilted measure induced by the drift, Jensen's inequality is applied to estimate the lower tail probability of the path possessing the prescribed drift. 

The Gaussian random variables $\xi(t,j)$ in \eqref{noise} are assumed to be independent, indexed by discrete time $t$ and spatial coordinate $j$. Let $\Lambda(a)$ denote the log-moment generating function of $\xi$, defined by
\begin{equation}
\label{logmom}
    \Lambda(a) := \log\left(\int_\R \exp(ay)dP(y)\right)=\frac{a^2}{2},
\end{equation} 
where the expectation is taken with respect to the standard Gaussian distribution $P(y):=P_T(\xi\leq y)$. Adding a constant drift merely induces a horizontal shift in the solution of \eqref{sol}. Motivated by \cite{mueller2022scaling}, we therefore introduce the tilted probability measure $\widehat{P}_T^{(a)}$, corresponding to a drift $a>4\pi^2 d J^{-1/2}$ in the direction of $\phi_1(j)$ in \eqref{kernel_notation}, through its Radon-Nikodym derivative with respect to $P_T$:
\begin{equation}
\label{RN}
    \frac{d\widehat{P}_T^{(a)}}{dP_T}=\prod_{t=1}^T\prod_{j=0}^{J-1}\exp(a\phi_1(j)\xi(t,j)-\Lambda(a\phi_1(j)))\text{ where $a>4\pi^2d J^{-1/2}$}.
\end{equation}
The next lemma gives a lower bound for $\frac{1}{T}\log Z_T$ in the different regimes of $\beta$.

\begin{lemma}\label{ZT} There exist positive constants $C,C',\tilde{d}$ and $\tilde{J}$ such that, for any $J>\tilde{J}$ and $\frac{\tilde{d}J}{\log J}\geq d\geq 1$, we have 
    \[\frac{1}{T}\log Z_T\geq \begin{cases}
            -C\beta^{\frac{2}{3}}d^{\frac{4-2\alpha}{3}}J^{\frac{1}{3}}(\log J)^{\frac{2}{3}} & \beta\geq C'd^{1+\alpha}J^{-2}(\log J)^{-1},\\
            -Cd^2J^{-1}&\beta<C'd^{1+\alpha}J^{-2}(\log J)^{-1}.\end{cases}\]
\end{lemma}
Before proving Lemma \ref{ZT}, we collect several preliminary facts. Let $\widehat{\E}$ and $\widehat{\text{Var}}$ denote expectation and variance under $\widehat{P}_T^{(a)}$, respectively. By construction,
\begin{equation*}
    \widehat{\E}[\xi(t,j)] = a\phi_1(j)\text{ and }\widehat{\text{Var}}[\xi(t,j)] = 1.
\end{equation*}

Applying Jensen's inequality to \eqref{penalF}, together with the definition of $Z_T$ in \eqref{Z} and the change of measure \eqref{RN}, yields
\begin{equation}
\begin{split}
\label{lowerBound}
     \frac{1}{T}\log Z_T &= \frac{1}{T}\log\widehat{\E}\left[\exp\left(-\frac{\beta}{d^\alpha}\sum_{t=1}^{T}N_d(t)-\log\frac{d\widehat{P}_T^{(a)}}{dP_T}\right)\right]\\
     &\geq -\frac{\beta}{Td^\alpha}\widehat{\E}\left[\sum_{t=1}^{T}N_d(t)\right]-\frac{1}{T}\widehat{\E}\left[\log\frac{d\widehat{P}_T^{(a)}}{dP_T}\right].
\end{split}
\end{equation}
The second term on the right-hand side is explicit. Indeed, using \eqref{logmom}, the linearity of expectation and independence, we obtain
\begin{equation*}
\begin{split}
    -\frac{1}{T}\sum_{t=1}^T\sum_{j=0}^{J-1}\left[a\phi_1(j)\widehat{\E}[\xi(t,j)]-\Lambda(a\phi_1(j))\right]=-\frac{a^2}{2}.
    \end{split}
\end{equation*}
Thus, the main task is to estimate the first term on the right-hand side of \eqref{lowerBound}. By the definition of $N_d(t)$ in \eqref{localTime}, this term can be rewritten as
\begin{equation*}
    -\frac{\beta}{Td^\alpha}\sum_{t=1}^{T}\sum_{0\leq i\neq j\leq J-1}\widehat{P}_T^{(a)}(\vert u(t,i)-u(t,j)\vert\leq d).
\end{equation*}
Following the strategy of \cite{mueller2022scaling}, we introduce the discrete pinned string
\begin{equation}
\label{pinned}
\begin{split}
    u_{t_0,n_0}(t,n) &:= \sum_{s=t_0}^{t-1}\sum_{m=0}^{J-1}G_{t-s}(n,m)\xi(s,m)+ \sum_{s=-\infty}^{t_0-1}\sum_{m=0}^{J-1}\left(G_{t-s}(n,m)-G_{t_0-s}(n_0,m)\right)\xi(s,m)\\
    &=:A(t_0,t,n,J)+B(t_0,n_0,t,n,J),
\end{split}
\end{equation}
which incorporates an infinite past, together with an anchoring at the reference point $n_0$ at time $t_0$. Note that the first term $A$ vanishes when $t=t_0$. Our first objective is to verify that the pinned string is well defined and that its spatial increments are stationary in time. Since $A$ is a finite sum, it remains only to check the convergence of $B$.

\begin{lemma}
    For any $T>0$,
    \begin{equation*}
        \max_{t_0\leq t\leq T}\max_{n,n_0\in\mathcal{J}}\widehat{\E}\left[B(t_0,n_0,t,n,J)^2\right]<\infty.
    \end{equation*}
\end{lemma}
\begin{proof}
Recall from \eqref{kernel_notation} that the kernel $G_t(n,m)$ admits the $L^2$-eigenfunction expansion
\begin{equation*}
    G_t(n,k) = \sum_{m=0}^{J-1}\rho_m^t\phi_m(n)\phi_m(k).
\end{equation*}
Since $\{\xi(s,m)\}_{s\in\mathcal{T},\,m\in\mathcal{J}}$ are independent standard Gaussians under $\widehat{P}_T^{(a)}$, the second moment of $B$ satisfies
    \begin{equation*}
    \begin{split}
        \widehat{\E}\left[B(t_0,n_0,t,n,J)^2\right] &= \sum_{s=-\infty}^{t_0-1}\sum_{m=0}^{J-1}\left(G_{t-s}(n,m)-G_{t_0-s}(n_0,m)\right)^2\\
        &=\sum_{s=-\infty}^{t_0-1}\sum_{m=0}^{J-1}\sum_{j=0}^{J-1}\sum_{k=0}^{J-1}\left[\rho_j^{t-s}\phi_j(m)\phi_j(n)-\rho_j^{t_0-s}\phi_j(m)\phi_j(n_0)\right]\\
        &\qquad\quad\cdot\left[\rho_k^{t-s}\phi_k(m)\phi_k(n)-\rho_k^{t_0-s}\phi_k(m)\phi_k(n_0)\right]\\
        &=\sum_{s=-\infty}^{t_0-1}\sum_{j=0}^{J-1}\sum_{k=0}^{J-1}\left(\sum_{m=0}^{J-1}\phi_j(m)\phi_k(m)\right)\left[\rho_j^{t-s}\phi_j(n)-\rho_j^{t_0-s}\phi_j(n_0)\right]\\
        &\qquad\quad\cdot\left[\rho_k^{t-s}\phi_j(n)-\rho_k^{t_0-s}\phi_j(n_0)\right]\\
        &=\sum_{s=-\infty}^{t_0-1}\sum_{j=0}^{J-1}\left[\rho_j^{t-s}\phi_j(n)-\rho_j^{t_0-s}\phi_j(n_0)\right]^2<\infty,
        \end{split}
    \end{equation*}
where we used the orthogonality of the eigenfunctions $\{\phi_j\}_{j\in\mathcal{J}}$ and $|\rho_j|<1$ for $j\geq 1$.
\end{proof}

\begin{lemma}(Shift invariance of the pinned string). For any fixed $t_0\in\mathcal{T}$ and $n_0\in\mathcal{J}$, under the measure $\widehat{P}_T^{(a)}$, the increment process
\begin{equation*}
    U_{t_0,n_0}(t,n_1,n_2):=u_{t_0,n_0}(t,n_1)-u_{t_0,n_0}(t,n_2), \quad t\in\mathcal{T},n_1,n_2\in\mathcal{J},
\end{equation*}
is stationary in time; i.e., for any $t_1,t_2\in\mathcal{T}$ and $t_0\leq t_1\leq t_2$,
\begin{equation*}
    \left(U_{t_0,n_0}(t_1,n_1,n_2)\right)\stackrel{D}{=}\left(U_{t_0,n_0}(t_2,n_1,n_2)\right).
\end{equation*}
\end{lemma}
\begin{proof}
Let $\theta_h$ denote the time-shift operator on the noise, defined by $(\theta_h\xi)(t,n) :=\xi(t+h, n)$. Using \eqref{pinned}, we can rewrite $U_{t_0,n_0}(t,n_1,n_2)$ as
\begin{equation*}
\begin{split}
    U_{t_0,n_0}(t,n_1,n_2)&= \sum_{s=t_0}^{t-1}\sum_{m=0}^{J-1}\left(G_{t-s}(n_1,m)-G_{t-s}(n_2,m)\right)\xi(s,m) \\
    &\quad\quad+ \sum_{s=-\infty}^{t_0-1}\sum_{m=0}^{J-1}\left(G_{t-s}(n_1,m)-G_{t-s}(n_2,m)\right)\xi(s,m)\\
    &=\sum_{s=-\infty}^{t-1}\sum_{m=0}^{J-1}\left(G_{t-s}(n_1,m)-G_{t-s}(n_2,m)\right)\xi(s,m).
\end{split}    
\end{equation*}
Shifting time by $h$ then gives
\begin{equation*}
\begin{split}
  U_{t_0,n_0}(t+h,n_1,n_2) &= \sum_{s=-\infty}^{t+h-1}\sum_{m=0}^{J-1}\left(G_{t+h-s}(n_1,m)-G_{t+h-s}(n_2,m)\right)\xi(s,m)\\
  &=\sum_{s'=-\infty}^{t-1}\sum_{m=0}^{J-1}\left(G_{t-s'}(n_1,m)-G_{t-s'}(n_2,m)\right)\xi(s'+h,m)\\
  &=\sum_{s'=-\infty}^{t-1}\sum_{m=0}^{J-1}\left(G_{t-s'}(n_1,m)-G_{t-s'}(n_2,m)\right)(\theta_h\xi)(s',m).
  \end{split}
\end{equation*}
Since the noise field $\{\xi(s,m)\}_{s\in\mathcal{T},\,m\in\mathcal{J}}$ is stationary Gaussian and $\theta_h$ preserves its law, the distribution of the increment process is invariant under time shifts. This proves the claim.
\end{proof}

From this point on, we write $u(t,n):=u_{t_0,n_0}(t,n)$ for simplicity. The stationarity of $u$ implies
\begin{equation}
\begin{split}
\label{total_penal}
    -\frac{\beta}{Td^\alpha}\sum_{t=1}^{T}\sum_{0\leq i\neq j\leq J-1}\widehat{P}_T^{(a)}(\vert u(t,i)-u(t,j)\vert\leq d) =-\frac{\beta}{d^\alpha}\sum_{0\leq i\neq j\leq J-1}\widehat{P}_T^{(a)}(\vert u(0,i)-u(0,j)\vert\leq d).
    \end{split}
\end{equation}
For $i\neq j$, the increment $u(0,i)-u(0,j)$ is Gaussian under $\widehat{P}_T^{(a)}$, and its mean and variance can be computed explicitly. By linearity of expectation, the Gaussianity of $\xi$, and Fubini's theorem (the series being absolutely convergent in $L^2$),
\begin{equation}
\begin{split}
\label{mean}
    \widehat{\E}[u(0,i)-u(0,j)]&=\E\left[\sum_{s=1}^\infty\sum_{n=0}^{J-1}[G_s(i,n)-G_s(j,n)]\xi(s,n)\right]\\
    &=\sum_{s=1}^\infty\sum_{n=0}^{J-1}[G_s(i,n)-G_s(j,n)]\widehat{\E}[\xi(s,n)]\\
    &=a\sum_{m=1}^{J-1}\left(\sum_{s=1}^\infty\rho_m^s\right)\left(\sum_{n=0}^{J-1}\phi_m(n)\phi_1(n)\right)\cdot\left[\phi_m(i)-\phi_m(j)\right]\\
    &=\frac{a\rho_1(\phi_1(i)-\phi_1(j))}{(1-\rho_1)}=a\mu_1(\phi_1(i)-\phi_1(j)),
\end{split}
\end{equation}
where $\mu_1 = \frac{\rho_1}{1-\rho_1}$. 

\begin{lemma}\label{varbound}(Variance estimate for spatial increments). For any distinct $i,j\in\mathcal{J}$, there exist constants $C,C'>0$ such that
    \begin{equation*}
        C\vert i-j\vert \leq \widehat{\text{Var}}[u(0,i)-u(0,j)]\leq C'\vert i-j\vert.
    \end{equation*}
\end{lemma}

\begin{proof}
Without loss of generality, assume $i>j$. Expanding the variance gives

\begin{equation}
\begin{split}
\label{var}
    \widehat{\text{Var}}[u(0,i)-u(0,j)]&=\sum_{s=1}^\infty\sum_{n=0}^{J-1}[G_s(i,n)-G_s(j,n)]^2\widehat{\text{Var}}[\xi(s,n)]\\
    &=\sum_{s=1}^\infty\sum_{n=0}^{J-1}\left[\sum_{m=1}^{J-1}\rho_m^s\phi_m(n)\left(\phi_m(i)-\phi_m(j)\right)\right]^2\\
    &=\sum_{s=1}^\infty\sum_{m=1}^{J-1}\rho_m^{2s}\left(\phi_m(i)-\phi_m(j)\right)^2\sum_{n=0}^{J-1}\phi^2_m(n)\\
    &=\frac{2}{J}\sum_{m=1}^{J-1}\frac{\cos^2\left(\frac{m\pi}{J}\right)}{1-\cos^2\left(\frac{m\pi}{J}\right)}\left[\cos\left(\frac{m\pi(i+0.5)}{J}\right)-\cos\left(\frac{m\pi(j+0.5)}{J}\right)\right]^2.
\end{split}
\end{equation}

We first derive the lower bound. Set
\begin{equation*}
    M:=\max\left\lbrace m\in\mathcal{J}:m<\frac{2J}{i-j}\right\rbrace.
\end{equation*}
Then for every $m=1,\dots,M$,
\begin{equation*}
    0<\frac{m\pi(i-j)}{2J}<\pi,
\end{equation*}
Moreover, since $(M+1)(i-j)\ge 2J$, we also have $M(i-j)\geq 2J-(i-j)\geq J+1$. Using the bound $\sin x\ge Cx$ on $(0,\pi)$, it follows that
\begin{equation*}
    \sin^2\left(\frac{m\pi(i-j)}{2J}\right)\geq C\frac{m^2(i-j)^2}{J^2}.
\end{equation*}
Substituting this into \eqref{var} yields
\begin{equation}
    \begin{split}
    \label{varlow}
    \widehat{\text{Var}}[u(0,i)-u(0,j)]&\geq CJ\sum_{m=1}^{J-1}\frac{\sin^2\left(\frac{m\pi(i-j)}{2J}\right)\sin^2\left(\frac{m\pi(i+j+1)}{2J}\right)}{m^2}\\
        &\geq \frac{C(i-j)^2}{J}\sum_{m=1}^{M}\sin^2\left(\frac{m\pi(i+j+1)}{2J}\right)\\
        &= \frac{C(i-j)^2}{J}\left[2M+1-\frac{\sin\left(\frac{(2M+1)(i+j+1)\pi}{2J}\right)}{\sin\left(\frac{(i+j+1)\pi}{2J}\right)}\right]\\
        &\geq \frac{C(i-j)^2}{J}\left[2M-\frac{1}{\sin\left(\frac{(i+j+1)\pi}{2J}\right)}\right].
    \end{split}
\end{equation}
By symmetry of the sine function around $\frac{\pi}{2}$, we may assume without loss of generality that $i+j<J-1$. Then
\begin{equation*}
    \sin\left(\frac{(i+j+1)\pi}{2J}\right)\geq\sin\left(\frac{(i-j)\pi}{2J}\right)\geq\frac{(i-j)}{J}.
\end{equation*}
Therefore,
\begin{equation}
\begin{split}
\label{sinM}
    \frac{1}{\sin\left(\frac{(i+j+1)\pi}{2J}\right)}\leq\frac{J}{i-j}\leq M.
\end{split}
\end{equation}
Combining \eqref{varlow} and \eqref{sinM}, we infer that
\begin{equation*}
    \begin{split}
        \widehat{\text{Var}}[u(0,i)-u(0,j)]&\geq \frac{C(i-j)^2M}{J}\geq C(i-j).
    \end{split}
\end{equation*}

To obtain the upper bound, note that the mean value theorem implies
\begin{equation*}
    \left[\cos\left(\frac{m\pi(i+0.5)}{J}\right)-\cos\left(\frac{m\pi(j+0.5)}{J}\right)\right]^2\leq 4\wedge \left(\frac{m\pi(i-j)}{J}\right)^2.
\end{equation*}
Substituting this estimate into \eqref{var} gives
\begin{equation*}
    \begin{split}
        \widehat{\text{Var}}[u(0,i)-u(0,j)]&\leq \frac{2}{J}\sum_{m=1}^{J-1}\frac{4\wedge \frac{m^2\pi^2}{J^2}(i-j)^2}{\sin^2\left(\frac{m\pi}{J}\right)}\leq \frac{C}{J}\sum_{m=1}^{J-1}\frac{J^2}{m^2}\wedge (i-j)^2\\
        &\leq \frac{C}{J}\left((i-j)^2+\int_1^{\frac{J}{i-j}}(i-j)^2dz + J^2\int_{\frac{J}{i-j}}^\infty\frac{1}{z^2}dz\right)\leq C(i-j),
    \end{split}
\end{equation*}
which completes the proof.
\end{proof}

For distinct indices $i\neq j$, \eqref{mean} together with Lemma \ref{varbound} yields
\begin{equation}
\label{penal}
    \widehat{P}^{(a)}_T(|u(0,i)-u(0,j)|\leq d)=C\int_{-d}^d\frac{1}{\sqrt{\vert i-j\vert}}\exp\left(\frac{-C'(z-a\mu_1(\phi_1(i)-\phi_1(j)))^2}{\vert i-j\vert}\right)dz,
\end{equation}
where $C$ and $C'$ are positive constants. We are now in a position to prove Lemma \ref{ZT}.
\begin{proof}[Proof of Lemma \ref{ZT}]
    Recall that $\rho_1=\cos\left(\frac{\pi}{J}\right)$, so
    \begin{equation*}
        \mu_1=\frac{\cos\left(\frac{\pi}{J}\right)}{1-\cos\left(\frac{\pi}{J}\right)}=\frac{1}{2\sin^2\left(\frac{\pi}{2J}\right)}-1.
    \end{equation*}
    For $0\leq x\leq\frac{\pi}{2}$, the bounds $\frac{2}{\pi}t\leq \sin t\leq t$ hold, and therefore
    \begin{equation*}
        \frac{J^2}{\pi^2}\leq \mu_1\leq \frac{J^2}{2}
    \end{equation*}
    whenever $J>\pi$. Set $\hat{J}=\left\lfloor\frac{J-1}{2}\right\rfloor$. If $0\leq i<\hat{J}\leq j\leq J-1$, then $0<j-i\leq J-1$ and $\frac{J}{2}<i+j+1<\frac{3J}{2}$. Using the trigonometric representation of $\phi_1(i)-\phi_1(j)$, one finds
    \begin{equation*}
        a\mu_1(\phi_1(i)-\phi_1(j))\geq \frac{\sqrt{2}}{\pi^2}aJ^{\frac{3}{2}}\sin\left(\frac{(i+j+1)\pi}{2J}\right)\sin\left(\frac{(j-i)\pi}{2J}\right)\geq \frac{1}{\pi^2}a(j-i)\sqrt{J}.
    \end{equation*}
    Now introduce the change of variable $z=\frac{1}{\pi^2}a\sqrt{J}w$. Since \eqref{RN} imposes $d\left(\frac{1}{\pi^2}a\sqrt{J}\right)^{-1}<\frac{1}{4}$, the Gaussian bound \eqref{penal} gives
    \begin{equation*}
    \begin{split}
        \widehat{P}^{(a)}_T&(|u(0,i)-u(0,j)|\leq d)\leq C'\int_{-d(Ca\sqrt{J})^{-1}}^{d(Ca\sqrt{J})^{-1}}\frac{a\sqrt{J}}{\sqrt{j-i}}\exp\left(-\frac{C''a^2J(w-j+i)^2}{j-i}\right)dw\\
        &\leq C'a\sqrt{J}\int_{-d(Ca\sqrt{J})^{-1}}^{d(Ca\sqrt{J})^{-1}}\exp\left(-C''a^2\left(1/4-j+i\right)^2\right)dw\leq Cd\exp\left(-C''a^2(j-i-1/4)^2\right).
    \end{split}
    \end{equation*}
Therefore, in the range $0\le i<\hat{J}\le j\le J-1$, the contribution to the penalization term in \eqref{total_penal} satisfies
    \begin{equation}
    \label{penal_1}
        \begin{split}
            \sum_{i=0}^{\hat{J}-1}\sum_{j=\hat{J}}^{J-1}\widehat{P}_T^{(a)}(\vert u(0,i)-u(0,j)\vert\leq d)&\leq Cd\sum_{k=1}^{J-1}k\exp\left(-C''a^2k^2\right)\\
            &\leq Cd\int_0^\infty k\exp\left(-C'a^2k^2\right)dk\leq \frac{Cd}{a^2},
        \end{split}
    \end{equation}
    where we used the substitution $k=j-i$, together with the bounds $\frac{3k}{4}\leq k-\frac{1}{4}\leq k$ for $k\geq 1$, and the fact that, in the truncated range $0\le i<\hat{J}\le j\le J-1$, there are at most $k$ pairs $(i,j)$ such that $j-i=k$. 
    
    Next consider the case $\hat{J}\le i<j\le J-1$. Then $J\leq i+j+1\leq2J-2$, which implies
    \begin{equation*}
        \sin\left(\frac{(i+j+1)\pi}{2J}\right)\geq1-\frac{i+j+1}{2J}.
    \end{equation*}
    Consequently,
    \begin{equation*}
        a\mu_1(\phi_1(i)-\phi_1(j))\geq CaJ^{\frac{3}{2}}\sin\left(\frac{(i+j+1)\pi}{2J}\right)\sin\left(\frac{(j-i)\pi}{2J}\right)\geq \frac{C''a}{\sqrt{J}}(j-i)(2J-(i+j+1)).
    \end{equation*}
    Now make the substitutions $z = C''aJ^{-1/2}w$, $u=J-i-1$ and $v=J-j-1$. Then \eqref{penal} yields
    \begin{equation}
    \label{penal_2}
        \begin{split}
            &\widehat{P}^{(a)}_T(|u(0,i)-u(0,j)|\leq d)\\
            &\leq C\int_{-d}^d\frac{1}{\sqrt{j-i}}\exp\left(\frac{-C'(z-C''aJ^{-1/2}(j-i)(2J-i-j-1))^2}{j-i}\right)dz\\
            &\leq \frac{Ca}{\sqrt{J}}\int_{-\frac{d\sqrt{J}}{C''a}}^{\frac{d\sqrt{J}}{C''a}}\frac{1}{\sqrt{u-v}}\exp\left(-\frac{C'a^2}{J(u-v)}(w-(u+v+1)(u-v))^2\right)dw\\
            &\leq \frac{CL}{\sqrt{u-v}}\int_{-\frac{d}{C''L}}^{\frac{d}{C''L}}\exp\left(-\frac{C'L^2}{u-v}(u-v-y)^2\right)dy,
        \end{split}
    \end{equation}
    where $w = (u+v+1)y$ and $L=\frac{a(u+v+1)}{\sqrt{J}}$. Consider the set $E:=\left\lbrace (u,v):u+v+1>\frac{4d\sqrt{J}}{C''a}\right\rbrace$. On $E$, one immediately has $-\frac{1}{4}<y<\frac{1}{4}$. Since $0\leq v<u\leq J-\hat{J}-1$, it follows that $u-v-w>u-v-\frac{1}{4}>0$. Hence the corresponding part of the penalization in \eqref{total_penal} can be estimated as
    \begin{equation*}
        \begin{split}
            &\sum_{v=0}^{J-2-\hat{J}}\sum_{u=v+1}^{J-1-\hat{J}}1_E\widehat{P}_T^{(a)}(\vert u(0,i)-u(0,j)\vert\leq d)\\
            &\leq \sum_{v=0}^{J-2-\hat{J}}\sum_{u=v+1}^{J-1-\hat{J}}\frac{CL}{\sqrt{u-v}}\int_{-\frac{d}{C''L}}^{\frac{d}{C''L}}\exp\left(-\frac{C'L^2}{u-v}(u-v-1/4)^2\right)dw\\
            &\leq \sum_{v=0}^{J-2-\hat{J}}\sum_{u=v+1}^{J-1-\hat{J}}\frac{Cd}{\sqrt{u-v}}\exp\left(-\frac{C'a^2(u+v+1)^2}{J(u-v)}(u-v-1/4)^2\right)\\
            &\leq\sum_{v=0}^{J-2-\hat{J}}\left(Cd\exp\left(-\frac{C'a^2(2v+1)^2}{J}\right)\right.\\
            &\quad\quad\left.+\int_{v+1}^{J-1-\hat{J}}\frac{Cd}{\sqrt{u-v}}\exp\left(-\frac{C'a^2(u+v+1)^2}{J(u-v)}(u-v-1/4)^2\right)du\right)\\
            &\leq \int_{-\infty}^{\infty}Cd\exp\left(-\frac{C'a^2(2v+1)^2}{J}\right)dv+\sum_{v=0}^{J-2-\hat{J}}\int_{3/4}^{J-5/4-\hat{J}-v}\frac{Cd}{\sqrt{y}}\exp\left(-\frac{C'a^2(2v+y)^2y}{J}\right)dy\\
            &\leq \frac{Cd\sqrt{J}}{a}+\int_{3/4}^\infty Cd\exp\left(-\frac{C'a^2y^2}{J}\right)dy+\int_{0}^{J}\int_{0}^{J}\frac{Cd}{\sqrt{y}}\exp\left(-\frac{C'a^2(2v+y)^2y}{J}\right)dydv,
        \end{split}
    \end{equation*}
where the term $u=v+1$ was separated in the third inequality, the substitution $y=u-v-1/4$ was used and the case $v=0$ was treated separately in the last inequality. Now set $g=y(\frac{a}{\sqrt{J}})^{2/3},h=2v(\frac{a}{\sqrt{J}})^{2/3}$. Since $aJ>1$, the previous display becomes
\begin{equation}
\label{penal_3}
    \begin{split}
        &\sum_{v=0}^{J-2-\hat{J}}\sum_{u=v+1}^{J-1-\hat{J}}1_E\widehat{P}_T^{(a)}(\vert u(0,i)-u(0,j)\vert\leq d)\\
        &\leq \frac{Cd\sqrt{J}}{a}+\frac{C'd\sqrt{J}}{a}\int_0^{2(aJ)^{2/3}}\left(\int_0^{(aJ)^{-4/3}}+\int_{(aJ)^{-4/3}}^{(aJ)^{2/3}}\right)\frac{1}{\sqrt{g}}\exp\left(-C''(h+g)^2g\right)dgdh\\
        &\leq \frac{Cd\sqrt{J}}{a} + \frac{C'd\sqrt{J}}{a}\left(\int_0^{2(aJ)^{2/3}}(aJ)^{-2/3}dh+\int_{(aJ)^{-4/3}}^{(aJ)^{2/3}}\frac{1}{\sqrt{g}}\int_{-\infty}^{\infty}\exp\left(-C''(h+g)^2g\right)dhdg\right)\\
        &\leq \frac{Cd\sqrt{J}}{a}+ \frac{C'd\sqrt{J}}{a}\int_{(aJ)^{-4/3}}^{(aJ)^{2/3}}\frac{1}{g}dg\leq\frac{Cd\sqrt{J}\log (J)}{a}.
    \end{split}
\end{equation}
It remains to estimate the contribution from $E^c$. Introduce the variables
\begin{equation*}
    n=u-v, m=u+v+1,D=\frac{1}{\sqrt{n}}\left[\frac{amn}{\sqrt{J}}+\frac{d}{C''}\right]\text{ and }D'=\frac{1}{\sqrt{n}}\left[\frac{amn}{\sqrt{J}}-\frac{d}{C''}\right].
\end{equation*}
With the further substitution $l=\frac{a(mn-w)}{\sqrt{Jn}}$, equation \eqref{penal_2} implies
\begin{equation*}
    \begin{split}
        \sum_{v=0}^{J-2-\hat{J}}\sum_{u=v+1}^{J-1-\hat{J}}1_{E^c}\widehat{P}_T^{(a)}(\vert u(0,i)-u(0,j)\vert\leq d)\leq\sum_{m=2}^{\lfloor\frac{4d\sqrt{J}}{C''a}\rfloor}\sum_{n=1}^{m-1}\int_{D'}^D\exp\left(-Cl^2\right)dl.
    \end{split}
\end{equation*}
Next define $F:=\left\lbrace (m,n):n\geq\left(1+\frac{d}{C''}\right)^2\frac{J}{a^2m^2}\right\rbrace$. For $(m,n)\in F$, one has $D'\ge 1$, and in particular $D'\geq\frac{am\sqrt{n}}{\sqrt{J}}(1+\frac{d}{C''})^{-1}$. This allows us to continue with
\begin{equation*}
    \begin{split}
        &\sum_{v=0}^{J-2-\hat{J}}\sum_{u=v+1}^{J-1-\hat{J}}1_{E^c}\widehat{P}_T^{(a)}(\vert u(0,i)-u(0,j)\vert\leq d)\\
        &\leq \sum_{m=2}^{\lfloor\frac{4d\sqrt{J}}{C''a}\rfloor}\sum_{n=1}^{m-1}\left(1_F\cdot(D-D')\exp\left(-CD'^2\right)+1_{F^c}\cdot(D-D')\right)\\
        &\leq \sum_{m=2}^{\lfloor\frac{4d\sqrt{J}}{C''a}\rfloor}\left(\int_{0}^{\infty}\frac{amd}{C''\left(1+\frac{d}{C''}\right)\sqrt{J}}\exp\left(-\frac{Ca^2m^2n}{\left(1+\frac{d}{C''}\right)^2J}\right)dn+\sum_{n=1}^{\lfloor\left(1+\frac{d}{C''}\right)^2\frac{J}{a^2m^2}\rfloor}\frac{d}{C''\sqrt{n}}\right)\\
        &\quad\quad +\sum_{m=2}^{\lfloor(\left(1+\frac{d}{C''}\right)^2\frac{J}{a^2})^{1/3}\rfloor}\sum_{n=1}^{m-1}\frac{d}{C''\sqrt{n}}\\
        &\leq \sum_{m=2}^{\lfloor\frac{4d\sqrt{J}}{C''a}\rfloor}\frac{Cd\left(1+\frac{d}{C''}\right)\sqrt{J}}{am} +\sum_{m=2}^{\lfloor(\left(1+\frac{d}{C''}\right)^2\frac{J}{a^2})^{1/3}\rfloor}C'd\sqrt{m}\leq \frac{Cd^2\sqrt{J}\log(J)}{a}+\frac{C'd^2\sqrt{J}}{a}.
    \end{split}
\end{equation*}
Combining this estimate with \eqref{penal_1} and \eqref{penal_3}, we arrive at
\begin{equation*}
    \sum_{0\leq i\neq j\leq J-1}\widehat{P}_T^{(a)}(\vert u(0,i)-u(0,j)\vert\leq d)\leq
        \frac{Cd^2\sqrt{J}\log(J)}{a}.
\end{equation*}
In order for this upper bound to remain consistent, it is enough to require $\frac{Cd^2\sqrt{J}\log(J)}{a}\leq J^2$, which is ensured by $d\leq\frac{\tilde{d}J}{\log J}$. Substituting the above estimate into \eqref{lowerBound}, we obtain
\begin{equation*}
\begin{split}
        \frac{1}{T}\log Z_T &\geq -\frac{C\beta d^{2-\alpha}\sqrt{J}\log(J)}{a}-\frac{a^2}{2},
\end{split}
\end{equation*}
The right-hand side is maximized at $\hat{a}=C\beta^{\frac{1}{3}}d^{\frac{2-\alpha}{3}}J^{\frac{1}{6}}(\log J)^{\frac{1}{3}}$, provided that $\hat{a}\geq 4\pi^2d J^{-\frac{1}{2}}$, which is equivalent to $\beta\geq Cd^{1+\alpha}J^{-2}(\log J)^{-1}$. If instead $\beta<Cd^{1+\alpha}J^{-2}(\log J)^{-1}$, then the optimum is attained at the boundary value $\hat{a}=4\pi^2d J^{-\frac{1}{2}}$. Substituting these choices of $\hat a$ into the previous display gives the desired lower bounds.
\end{proof}

\section{Small distance tail estimate}
\label{Small distance tail estimate}
In this section, we establish an upper bound on the lower tail of $R(T,J)$, showing that this function rarely takes values that are too small when $T$ is large. Specifically, we prove the following result:
\begin{prop}
\label{lower}
    There exist positive constants $\tilde{d}, \tilde{J}, \beta_C$ and $K_1$, such that for $J>\tilde{J}$, $\beta \geq\beta_Cd^{4+\alpha}J^{-2}(\log J)^{2}$ and $\frac{\tilde{d}J}{\log J}\geq d\geq 1$,
    \begin{equation*}
    \lim_{T\to \infty}Q_T\left[R(T,J)\leq K_1\beta_C^{\frac{1}{3}}d J\right]=0.
\end{equation*}
\end{prop}
\begin{proof}
Fix $K>0$ and define the event $A_{K,T}:=\{R(T,J)\leq K\}$. By the definition of $R(T,J)$ in \eqref{RTJ}, the event $A_{K,T}$ implies
\begin{equation}
\label{inclusion1}
    \left\vert \left\lbrace t\in \mathcal{T}:\frac{1}{J}\sum_{n=0}^{J-1}(u(t,n)-\bar{u}(t))^2\leq2K^2\right\rbrace\right\vert\geq \frac{T}{2},
\end{equation}
where $|\cdot|$ denotes cardinality and $\bar u(t)$, defined in \eqref{ubar}, is the spatial average of $u(t,\cdot)$ at time $t$. Next observe that if
\begin{equation*}
    \frac{1}{J}\sum_{n=0}^{J-1}(u(t,n)-\bar{u}(t))^2\leq2K^2,
\end{equation*}
then a simple counting argument shows that
\begin{equation}
    \label{inclusion2}
    \left\vert \left\lbrace n\in \mathcal{J}:u(t,n)\in[\bar{u}(t)-2K,\bar{u}(t)+2K]\right\rbrace\right\vert\geq \frac{J}{2}.
\end{equation}
Set $d_t^{\pm}=\bar{u}(t)\pm 2K$, so that $d_t^+-d_t^-=4K$. There exist integers $z^+,z^-\in\mathbb{Z}$ and a parameter $\gamma\in[0,1]$ such that
\begin{equation*}
    d_t^+=z^+d-\gamma d, \quad d_t^-\in(z^-d-\gamma d,z^-d+(1-\gamma)d].
\end{equation*}
To measure the local density of the profile $u(t,\cdot)$ at time $t$, we define the number of spatial sites that fall within a given interval of length $d$ as
\begin{equation}
\label{localT}
    l_{d,t}(z)=\sum_{i=0}^{J-1}1_{u(t,i)\in(zd-\gamma d,zd+(1-\gamma)d]}.
\end{equation}

Combining \eqref{inclusion1} and \eqref{inclusion2}, we see that on the event $A_{K,T}$,
\begin{equation}
\label{localbound}
    \left\vert \left\lbrace t\in \mathcal{T}:\sum_{z=z^-}^{z^+-1}l_{d,t}(z)\geq \frac{J}{2}\right\rbrace\right\vert\geq \frac{T}{2}.
\end{equation}
Note that even if two values $u(t,i)$ and $u(t,j)$ satisfy $\vert u(t,i)-u(t,j)\vert<d$, they need not belong to the same interval in the partition \eqref{localT}. Nevertheless, the following lower bound remains valid:
\begin{equation}
\label{localTineq}
    N_d(t)\geq\sum_{z\in\Z}l^2_{d,t}(z)-J,
\end{equation}
where $N_d(t)$ is the approximate self-intersection count defined in \eqref{localTime}.

Applying the Cauchy-Schwarz inequality to the sum over $z$, and then using \eqref{localbound} and \eqref{localTineq}, yields on $A_{K,T}$
\begin{equation}
\begin{split}
\label{QmeasB}
\sum_{t=1}^{T}N_d(t)&\geq\sum_{t=1}^{T}\sum_{z\in\Z}l^2_{d,t}(z)-TJ\geq\sum_{t=1}^{T}\sum_{z=z^-}^{z^+-1}l^2_{d,t}(z)-TJ\\
    &\geq \frac{1}{z^+-z^-}\sum_{t=1}^{T}\left(\sum_{z=z^-}^{z^+-1}l_{d,t}(z)\right)^2-TJ\\
    &\geq \frac{d}{4K+d}\cdot\frac{T}{2}\cdot\left(\frac{J}{2}\right)^2-TJ = TJ\left(\frac{d J}{32K+8d}-1\right).
    \end{split}
\end{equation}
At this point it is useful to note that $K$ can be at most a constant multiple of $dJ$. Indeed, otherwise one could take $u(t,i)=id$, in which case the radius of gyration is of order $dJ<K$, while $N_d(t)=0$ for every $t$. In the continuous model, the diagonal set has zero measure, whereas in the discrete model the diagonal consists of exactly $J$ points and therefore contributes nontrivially to the counting measure. To obtain a lower bound for $K$ of at least order $J$, one must therefore take $d>1$ sufficiently large so that the diagonal contribution does not overwhelm the off-diagonal self-interaction terms.

We now estimate the logarithmic asymptotic probability under $Q_T$. Using Lemma \ref{ZT}, together with \eqref{QmeasB} and the definition of the penalizing factor $\mathcal{E}_T$ in \eqref{Qmeas}, we find
\begin{equation}
\label{smalldistance}
\begin{split}
    &\limsup_{T\to \infty}\frac{1}{T}\log Q_T\left[R(T,J)\leq K\right]\\
    &\leq\limsup_{T\to \infty}\frac{1}{T}\log \E^{P_T}\left[\mathcal{E}_T1_{R(T,J)\leq K}\right]-\liminf_{T\to\infty}\frac{1}{T}\log Z_T\\
    &\leq\begin{cases}
            -\frac{\beta J}{d^{\alpha}}\left(\frac{d J}{32K+8d}-1\right)+C\beta^{\frac{2}{3}}d^{\frac{4-2\alpha}{3}}J^{\frac{1}{3}}(\log J)^{\frac{2}{3}} & \beta\geq C'd^{1+\alpha}J^{-2}(\log J)^{-1},\\
            -\frac{\beta J}{d^{\alpha}}\left(\frac{d J}{32K+8d}-1\right)+Cd^2J^{-1}&\beta<C'd^{1+\alpha}J^{-2}(\log J)^{-1}.\end{cases}
    \end{split}
\end{equation}
If $\beta<C'd^{1+\alpha}J^{-2}(\log J)^{-1}$, then $\frac{\beta J}{d^\alpha}<C'd J^{-1}(\log J)^{-1}<Cd^2J^{-1}$, so for large $J$ the upper bound of \eqref{smalldistance} cannot be negative. By contrast, if $\beta \geq\beta_Cd^{4+\alpha}J^{-2}(\log J)^{2}$ and we choose $K=K_1\beta_C^{\frac{1}{3}}d J$, then \eqref{smalldistance} becomes
\begin{equation*}
    \limsup_{T\to \infty}\frac{1}{T}\log Q_T\left[R(T,J)\leq K_1\beta_C^{\frac{1}{3}}d J\right]\leq \left(-\frac{\beta_C^{1/3}d J}{32K_1\beta_C^{\frac{1}{3}}d J+8d}+C'\right)\beta^{\frac{2}{3}}d^{\frac{4-2\alpha}{3}}J^{\frac{1}{3}}(\log J)^{\frac{2}{3}}.
\end{equation*}
Choosing $K_1$ sufficiently small makes the coefficient in parentheses strictly negative. Hence the exponential decay rate is negative, and therefore
\begin{equation*}
    \lim_{T\to \infty}Q_T\left[R(T,J)\leq K_1\beta_C^{\frac{1}{3}}d J\right]=0,
\end{equation*}
which completes the proof.
\end{proof}

\section{Large distance tail estimate}
\label{Large distance tail estimate}
In this section, we establish an asymptotic upper bound for the tail probability of the quantity $R(T,J)$ defined in \eqref{RTJ}. The result is formalized in the following proposition.
\begin{prop}
\label{upper}
    There exist positive constants $\tilde{d}, \tilde{J}, \beta_C$ and $K_2$, such that for $J>\tilde{J}$, $\beta \geq\beta_Cd^{4+\alpha}J^{-2}(\log J)^{2}$ and $\frac{\tilde{d}J}{\log J}\geq d\geq 1$, we have
\begin{equation*}
    \lim_{T\to\infty}Q_T\left[R(T,J)\geq K_2\beta^{\frac{1}{3}}d^{\frac{2-\alpha}{3}}J^{\frac{5}{3}}(\log J)^{-\frac{1}{6}}\right] = 0.
\end{equation*}
\end{prop}

To prove this result, we begin by analyzing the structure of $\bar{u}(t)$ in \eqref{ubar}. Writing
\begin{equation*}
    \begin{split}
        \bar{u}(t) &= \frac{1}{J}\sum_{n=0}^{J-1}u(t,n)=\frac{1}{J}\sum_{n=0}^{J-1}\sum_{k=0}^{J-1}G_t(n,k)u_0(k) + \frac{1}{J}\sum_{n=0}^{J-1}\sum_{s=0}^{t-1}\sum_{k=0}^{J-1}G_{t-s}(n,k)\xi(s,k)=:\bar{S}_1 + \bar{S}_2,
    \end{split}
\end{equation*}
we isolate the deterministic and stochastic contributions. Using trigonometric identities together with the orthogonality of the eigenfunctions, one checks that
\begin{equation*}
    \sum_{n=0}^{J-1}\phi_m(n)=0
\end{equation*}
for every $m\ge 1$. Hence only the mode $m=0$ contributes to the spatial average. It follows that
\begin{equation*}
        \bar{S}_1=\frac{1}{J}\sum_{n=0}^{J-1}\sum_{k=0}^{J-1}u_0(k)=\sum_{k=0}^{J-1}u_0(k) \text{ and }
        \bar{S}_2=\frac{1}{J}\sum_{n=0}^{J-1}\sum_{s=0}^{t-1}\sum_{k=0}^{J-1}\xi(s,k)=\sum_{s=0}^{t-1}\sum_{k=0}^{J-1}\xi(s,k).
\end{equation*}
Substituting these expressions into the full solution \eqref{sol}, we obtain the fluctuation around the mean:
\begin{equation*}
    u(t,n)-\bar{u}(t) = \sum_{k=0}^{J-1}\sum_{m=1}^{J-1}\rho_m^t\phi_m(n)\phi_m(k)u_0(k) + \sum_{s=0}^{t-1}\sum_{k=0}^{J-1}\sum_{m=1}^{J-1}\rho_m^{t-s}\phi_m(n)\phi_m(k)\xi(s,k).
\end{equation*}
Since $|\rho_m|<1$ for all $m\ge 1$, the deterministic term involving $u_0$ converges to zero uniformly in $n$ as $t\to\infty$. Thus the long-time behavior of $R(T,J)$ is determined entirely by the stochastic part. For this reason, we may assume without loss of generality that $u_0\equiv 0$, in which case
\begin{equation*}
\begin{split}
    u(t,n)-\bar{u}(t)&=\sum_{s=0}^{t-1}\sum_{k=0}^{J-1}\sum_{m=1}^{J-1}\rho_m^{t-s}\phi_m(n)\phi_m(k)\xi(s,k)=\sum_{m=1}^{J-1}X_t^{(m)}\phi_m(n),
    \end{split}
\end{equation*}
where
\begin{equation*}
    X_t^{(m)}:=\sum_{s=0}^{t-1}\rho_m^{t-s}\sum_{k=0}^{J-1}\phi_m(k)\xi(s,k)
\end{equation*}
satisfies the recursion
\begin{equation*}
    X_{t+1}^{(m)} = \rho_mX_t^{(m)} + \tilde{\xi}_t^{(m)},\quad X_0^{(m)} = 0,
\end{equation*}
with $\tilde{\xi}_t^{(m)}:=\rho_m\sum\limits_{k=0}^{J-1}\phi_m(k)\xi(t,k)$ an i.i.d Gaussian random variable of variance $\sigma_m^2$. Therefore, for each $1\le m\le J-1$, the process $X_t^{(m)}$ is an autoregressive process of order one (AR(1)). Using the orthonormality of $\{\phi_m(n)\}_{m\in\mathcal{J},\,n\in\mathcal{J}}$, we arrive at
\begin{equation}
\label{st}
    R(T,J)^2=\frac{1}{TJ}\sum_{t=1}^T\sum_{n=0}^{J-1}(u(t,n)-\bar{u}(t))^2 = \frac{1}{J}\sum_{m=1}^{J-1}\frac{1}{T}\sum_{t=1}^T\left(X_t^{(m)}\right)^2.
\end{equation}

We now invoke the large deviation principle from \cite{bryc1993large} for the quadratic functional $S_T^{(m)}:=\frac{1}{T}\sum\limits_{t=1}^T\left(X_t^{(m)}\right)^2$ of the autoregressive process.
\begin{lemma}[\cite{bryc1993large}]
\label{ldp}
    If $\vert\rho_m\vert<1$, then $\left\lbrace S^{(m)}_T\right\rbrace_{m\in\mathcal{J}}$ satisfies a large deviation principle with good rate function
    \begin{equation*}
        I_m(x)=\begin{cases}
           -\frac{1}{2}\ln\left(\frac{2x}{\sigma_m^2+\sqrt{4\rho_m^2x^2+\sigma_m^4}}\right)+\frac{1}{2\sigma_m^2}\left[(\rho_m^2+1)x-\sqrt{4\rho_m^2x^2+\sigma_m^4}\right]&\text{for } x>0,\\
            \infty&\text{for } x\leq 0.
        \end{cases}
    \end{equation*}
\end{lemma}
\begin{proof}
For fixed $m$, write $X_t:=X_t^{(m)}$ and $S_T:=S_T^{(m)}$ to simplify notation. Integrating the Gaussian density shows that for every $\lambda<\frac{1}{2\sigma_m^2}$,
    \begin{equation*}
    \E\left[\exp\left(\lambda X^2_{t+1}\right)\vert X_{t}\right] = \frac{1}{\sqrt{1-2\lambda\sigma_m^2}}\exp\left(\frac{\lambda\rho_m^2X_t^2}{1-2\lambda\sigma_m^2}\right).
    \end{equation*}
    Iterating conditional expectations and using $X_0=0$, one obtains the moment-generating function of $S_T$
    \begin{equation*}
        \E[\exp(yTS_T)]=\prod_{k=0}^{T-1}(1-2\sigma_m^2\lambda_k)^{-1/2},
    \end{equation*}
    where the sequence $\{\lambda_k\}_{k\ge 0}$ is defined recursively by
    \begin{equation*}
        \lambda_{k+1}=\frac{\rho_m^2\lambda_k}{1-2\lambda_k\sigma_m^2} + y,
    \end{equation*}
    with an initial condition $\lambda_0=y<\frac{1}{2\sigma_m^2}$. The critical threshold is $\delta=\frac{(1-\vert \rho_m\vert)^2}{2\sigma_m^2}$. For $y<\delta$, the associated M\"obius map, represented by the matrix $M = \begin{bmatrix}
    \rho_m^2-2\sigma_m^2y & y\\
    -2\sigma_m^2 & 1\\
    \end{bmatrix}
    $
    admits a limiting function
    \[g(y)=\frac{1-\rho_m^2+2\sigma_m^2y-\sqrt{(1-\rho_m^2+2\sigma_m^2y)^2-8\sigma_m^2y}}{4\sigma_m^2} \text{ and } L(y)=-1/2\ln(1-2\sigma_m^2g(y)).\]
    Finally, introducing the variable $z=\sqrt{(1-\rho_m^2+2\sigma_m^2y)^2-8\sigma_m^2y}$ solving the equation $\frac{\partial L(y)}{\partial y}=x$, one finds that $z=\sigma_m^2/x$. Substituting this into the Legendre transform $I_m(x)=xy(x)-L(y(x))$ gives the stated closed form for $I_m(x)$.
\end{proof}
\begin{remark}
\label{ldpcond}
	The rate function $I_m(x)$ vanishes at $x = \frac{\sigma_m^2}{1-\rho_m^2}$, which is the ergodic mean of $S_T^{(m)}$ by the ergodic theorem, and is strictly increasing for $x > \frac{\sigma_m^2}{1-\rho_m^2}$. This is precisely the regime relevant for upper deviations.
\end{remark}

Since $\sigma_m^2$, the variance of $\tilde{\xi}_t^{(m)}$, is equal to $\rho_m^2$, we may set $x=\frac{\rho_m^2k}{1-\rho_m^2}$. Substituting this into the rate function yields
\begin{align*}
    I_{m}(x(k))&=-\frac{1}{2}\ln\left(\frac{\frac{2\rho_m^2k}{1-\rho_m^2}}{\rho^2_m+\sqrt{\frac{4\rho_m^6k^2}{(1-\rho_m^2)^2}+\rho_m^4}}\right)+\frac{1}{2\rho_m^2}\left[\frac{(1+\rho_m^2)\rho_m^2k}{1-\rho_m^2}-\sqrt{\frac{4\rho_m^6k^2}{(1-\rho_m^2)^2}+\rho_m^4}\right]\\
    &=-\frac{1}{2}\ln\left(\frac{2}{1-\rho_m^2+\sqrt{(1-\rho_m^2)^2+4\rho_m^2k^2}}\right)+\frac{(k^2-1)(1-\rho_m^2)}{2((\rho_m^2+1)k+\sqrt{(1-\rho_m^2)^2+4\rho_m^2k^2})}\\
    &>\frac{(k-1)^2(1-\rho_m)}{2k(1+\rho_m)}>\frac{k}{4J^2}
\end{align*}
for sufficiently large $k$. Consequently, Lemma \ref{ldp} gives
\begin{equation}
\label{P_rate_bound}
    \limsup_{T\to\infty}\frac{1}{T}\log P\left(S^{(m)}_T>\frac{\rho_m^2k}{1-\rho_m^2}\right)\leq -I_m\left(\frac{\rho_m^2k}{1-\rho_m^2}\right)<\frac{-k}{4J^2}.
\end{equation}
We now sum over $m$ to control the tail of $R(T,J)$. Define the normalizing constant
\begin{equation}
\label{c0}
    c_0(J)=\left(\sum_{m=1}^{J-1}\frac{\rho_m^2}{{1-\rho_m^2}}\right)^{-1}= \left[\sum_{m=1}^{J-1}\frac{\cos^2\left(\frac{m\pi}{J}\right)}{1-\cos^2\left(\frac{m\pi}{J}\right)}\right]^{-1}>\left[2\sum_{m=1}^{\lfloor (J-1)/2\rfloor}\frac{J^2}{m^2\pi^2}\right]^{-1}>\frac{3}{J^2}.
\end{equation}
Therefore, whenever $K^2Jc_0(J)$ is sufficiently large, \eqref{st}, \eqref{P_rate_bound}, and \eqref{c0} imply
\begin{equation*}
\begin{split}
\label{tailP}
        &\limsup_{T\to\infty}P[R(T,J)\geq K] \\
        &= \limsup_{T\to\infty}P\left[R(T,J)^2\geq K^2\right]=\limsup_{T\to\infty}P\left(\sum_{m=1}^{J-1}S_T^{(m)}\geq JK^2\right)\\
    &\leq \limsup_{T\to\infty}\sum_{m=1}^{J-1}P\left(S_T^{(m)}\geq\frac{\rho_m^2Jc_0(J)K^2}{1-\rho_m^2}\right)\leq\limsup_{T\to\infty}\sum_{m=1}^{J-1}\exp\left(-\frac{TI_m\left(\frac{\rho_m^2Jc_0(J)K^2}{1-\rho_m^2}\right)}{2}\right)\\
    &\leq \limsup_{T\to\infty}(J-1)\exp\left(-\frac{3TK^2}{4J^{3}}\right).
\end{split}
\end{equation*}

Now choose $K = K_2\beta^{\frac{1}{3}}d^{\frac{2-\alpha}{3}}J^{\frac{5}{3}}(\log J)^{-\frac{1}{6}}$. Then, for $\beta \geq\beta_Cd^{4+\alpha}J^{-2}(\log J)^{2}$,
\begin{equation*}
\begin{split}
    \limsup_{T\to \infty}\frac{1}{T}\log Q_T\left[R(T,J)\geq K\right]&\leq\limsup_{T\to \infty}\frac{1}{T}\log \E^{P_T}\left[1_{R(T,J)\geq K}\right]-\liminf_{T\to\infty}\frac{1}{T}\log Z_T\\
    &\leq -\frac{CK^2\log J}{J^3}+C'\beta^{\frac{2}{3}}d^{\frac{4-2\alpha}{3}}J^{\frac{1}{3}}(\log J)^{\frac{2}{3}}\\
    &\leq (-CK_2^2+C')\beta^{\frac{2}{3}}d^{\frac{4-2\alpha}{3}}J^{\frac{1}{3}}(\log J)^{\frac{2}{3}}.
    \end{split}
\end{equation*}
Choosing $K_2$ sufficiently large makes the coefficient strictly negative. Hence the exponential decay rate is negative, and therefore
\begin{equation*}
\lim_{T\to\infty}Q_T\left[R(T,J)\geq K_2\beta^{\frac{1}{3}}d^{\frac{2-\alpha}{3}}J^{\frac{5}{3}}(\log J)^{-\frac{1}{6}}\right] = 0,
\end{equation*}
which completes the proof. 

\section{The limiting renormalized regime}
\label{limitalpha}
We impose $d^{4+\alpha}J^{-2}(\log J)^2=1$ to make $\beta$ independent of both $J$ and $d$, that is $d=\left(\frac{J}{\log J}\right)^{\frac{2}{4+\alpha}}$. Under this choice, define the event 
\begin{equation*}
    \mathcal{G}(\alpha,T):=\left\lbrace K_1\beta_C^{\frac{1}{3}}J^{\frac{6+\alpha}{4+\alpha}}(\log J)^{-\frac{2}{4+\alpha}}\leq R(T,J)\leq K_2\beta^{\frac{1}{3}}J^{\frac{8+\alpha}{4+\alpha}}(\log J)^{\frac{\alpha-4}{2\alpha+8}}\right\rbrace.
\end{equation*} 
By Propositions \ref{lower} and \ref{upper}, for $\beta\ge \beta_C$ and all sufficiently large $J$,
\begin{equation*}
    \lim_{T\to\infty}Q_T[\mathcal{G}(\alpha,T)]=1.
\end{equation*}
The next lemma shows that the limits in $\alpha$ and $T$ may be interchanged, even though both the measure $Q_{T}=Q_{\alpha,T}$ and event $\mathcal{G}(\alpha, T)$ depend on $\alpha$. The limiting measure is characterized by the penalizing factor in \eqref{limit_penal}, thereby completing the proof of Theorem \ref{mainTh}.

\begin{lemma}
\label{interchange}
    There exist positive constants $\beta_C$ and $\tilde{J}$ such that for $\beta\geq\beta_C$ and $J>\tilde{J}$, we have
    \begin{equation*}
    \lim_{T\to\infty}\lim_{\alpha\to\infty}Q_{\alpha,T}[\mathcal{G}(\alpha,T)]=\lim_{\alpha\to\infty}\lim_{T\to\infty}Q_{\alpha,T}[\mathcal{G}(\alpha,T))] = 1,
    \end{equation*}
    where 
    \begin{equation*}
        \lim_{\alpha\to\infty}Q_{\alpha,T}[\mathcal{G}(\alpha,T)]=\widetilde{Q}_T\left[K_1\beta_C^{\frac{1}{3}}J\leq R(T,J)\leq K_2\beta^{\frac{1}{3}}J(\log J)^{\frac{1}{2}}\right]
    \end{equation*}
    with penalizing factor
    \begin{equation*}
    \widetilde{\mathcal{E}}_T=\exp\left(-\frac{\beta(\log J)^2}{J^2}\sum_{t=1}^TN_1(t)\right).
    \end{equation*}
\end{lemma}
\begin{proof}
We first show that the convergence in $T$ is uniform in $\alpha$. By Propositions \ref{lower} and \ref{upper}, there exists $T_0$ such that for all $T\ge T_0$,
\begin{equation*}
\begin{split}
    1-Q_{\alpha,T}{[\mathcal{G}(\alpha,T)]}&\leq \exp\left(-C(K_1)T\beta^{\frac{2}{3}}d^{\frac{4-2\alpha}{3}}J^{\frac{1}{3}}(\log J)^{\frac{2}{3}}\right)+\exp\left(-C'(K_2)T\beta^{\frac{2}{3}}d^{\frac{4-2\alpha}{3}}J^{\frac{1}{3}}(\log J)^{\frac{2}{3}}\right)\\
    &\leq C\exp\left(-C'(K_1,K_2)T\beta_C^{\frac{2}{3}}J^{\frac{4-\alpha}{4+\alpha}}(\log J)^{\frac{2\alpha}{4+\alpha}}\right),
\end{split}
\end{equation*}
which implies
\begin{equation*}
    \lim_{T\to\infty}\sup_{\alpha\geq 0}\vert 1-Q_{\alpha,T}{[\mathcal{G}(\alpha,T)]}\vert\leq\lim_{T\to\infty}C\exp\left(-C'(K_1,K_2,\beta_C)TJ^{-1}(\log J)^{2}\right)=0.
\end{equation*}
Next set $\mathcal{G}(\infty,T):=\left\lbrace K_1\beta_C^{\frac{1}{3}}J\leq R(T,J)\leq K_2\beta^{\frac{1}{3}}J(\log J)^{\frac{1}{2}}\right\rbrace$ and define the limiting measure $Q_{\infty,T}$ by the penalizing factor 
\begin{equation*}
    \mathcal{E}_{T,\infty}=\exp\left(-\frac{\beta(\log J)^2}{J^2}\sum_{t=1}^TN_1(t)\right),
\end{equation*}
with partition function $Z_{T,\infty}=\mathbb{E}^{P_{T,J}}[\mathcal{E}_{T,\infty}]$. 

For every fixed polymer configuration $u(t,n)$, $\lim_{\alpha\to\infty}N_{d(\alpha,J)}(t)=N_1(t)$ since $d(\alpha,J)\to 1$ and the map $d \mapsto N_d(t)$ is right-continuous with left limits, and $P_T$-almost surely no pair satisfies $|u(t,i)-u(t,j)| = 1$ exactly (as $u(t,i)-u(t,j)$ has a continuous Gaussian distribution under $P_T$). Therefore $\mathcal{E}_{T,\alpha}$ converges to $\mathcal{E}_{T,\infty}$ $P_T$-almost surely. Since $0\le \mathcal{E}_{T,\alpha}\le 1$, the bounded convergence theorem gives
\begin{equation*}
    \lim_{\alpha\to\infty}\Vert \mathcal{E}_{T,\alpha}-\mathcal{E}_{T,\infty}\Vert_{L^1(P_T)} =0 .
\end{equation*}

We now compare $Q_{\alpha,T}$ and $Q_{\infty,T}$ in total variation. Recall that
\begin{equation*}
    \Vert Q_{\alpha,T}-Q_{\infty,T}\Vert_{TV} = \sup_{A}\vert Q_{\alpha,T}(A)-Q_{\infty,T}(A)\vert,
\end{equation*}
where the supremum is taken over all measurable sets $A$. Since both measures are absolutely continuous with respect to $P_T$, with Radon-Nikodym derivatives $\frac{d Q_{T,\alpha}}{dP_T} = \frac{\mathcal{E}_{T,\alpha}}{Z_{T,\alpha}}$ and $\frac{dQ_{T,\infty}}{dP_T} = \frac{\mathcal{E}_{T,\infty}}{Z_{T,\infty}}$. Taking $L^1(P_T)$ norms and applying the triangle inequality yield
\begin{equation*}
\begin{split}
    \lim_{\alpha\to\infty}\Vert Q_{\alpha,T}-Q_{\infty,T}\Vert_{TV}& =\lim_{\alpha\to\infty}\left\Vert \frac{\mathcal{E}_{T,\alpha}}{Z_{T,\alpha}}-\frac{\mathcal{E}_{T,\infty}}{Z_{T,\infty}}\right\Vert_{L^1(P_T)}\\
    &\leq \lim_{\alpha\to\infty}\frac{1}{Z_{T,\alpha}}\Vert \mathcal{E}_{T,\alpha}-\mathcal{E}_{T,\infty}\Vert_{L^1(P_T)}+\frac{\Vert \mathcal{E}_{T,\alpha}-\mathcal{E}_{T,\infty}\Vert_{L^1(P_T)}}{Z_{T,\alpha}Z_{T,\infty}}\Vert\mathcal{E}_{T,\infty}\Vert_{L^1(P_T)}=0,
    \end{split}
\end{equation*}
which implies $\lim_{\alpha\to\infty}Q_{\alpha,T}(\mathcal{G}(\cdot,T)) = Q_{\infty,T}(\mathcal{G}(\cdot,T))$ since $\mathcal{G}(\cdot,T)$ is measurable for each $T$.

It remains to identify the limit of the events $\mathcal{G}(\alpha,T)$. For each fixed realization of $R(T,J)$, the indicator $1_{\mathcal{G}(\alpha,T)}$ fails to converge to $1_{\mathcal{G}(\infty,T)}$ only if $R(T,J)$ is exactly at a boundary point. Since $R(T,J)$ is continuous and $Q_{\infty,T}$ is absolutely continuous with respect to $P_T$, $1_{\mathcal{G}(\alpha,T)}$ converges to $1_{\mathcal{G}(\infty,T)}$ $Q_{\infty,T}$-almost surely. Then, under the fixed measure $Q_{\infty,T}$, the bounded convergence theorem yields
\begin{equation*}
    \lim_{\alpha\to\infty}Q_{\infty,T}[\mathcal{G}(\alpha,T)]=Q_{\infty,T}[\mathcal{G}(\infty,T)]. 
\end{equation*}
Combining the convergence of the measures with the convergence of the events, we conclude that for each fixed $T$,
\[
\lim_{\alpha\to\infty}Q_{\alpha,T}[\mathcal{G}(\alpha,T)]
=
Q_{\infty,T}[\mathcal{G}(\infty,T)].
\]
Together with the uniform convergence in $T$ established at the beginning of the proof, this justifies the interchange of limits and completes the proof.
\end{proof}

\subsection*{Data availability} This is a theoretical paper with no associated data.

\section*{Declaration}
\subsection*{Conflict of interest} No funding was received for conducting this study.

\bibliographystyle{alpha}
\bibliography{polymer}

\end{document}